\documentclass[10pt,a4paper]{amsart}
          
\usepackage{mystyle}

\title{ Chow groups and pseudoeffective cones of complexity one \textit{T}-varieties}
\author{Bernt Ivar Utst\o l N\o dland}

\begin{document} 
\begin{abstract}
We show that the pseudoeffective cone of $k$-cycles on a complete complexity one $T$-variety is rational polyhedral for any $k$, generated by classes of $T$-invariant subvarieties. When $X$ is also rational, we give a presentation of the Chow groups of $X$ in terms of generators and relations, coming from the combinatorial data defining $X$ as a $T$-variety. 
\end{abstract}

\maketitle

\section{Introduction}

A $T$-variety is a normal complex algebraic variety $X$ with an effective action of an algebraic torus $T$. The complexity of a $T$-variety is defined as $\dim X -\dim T$, thus $T$-varieties of complexity zero correspond to the  toric varieties. For toric varieties there is a well-known correspondence between the geometry of a variety and combinatorial data coming from the $T$-action. There has in recent years been developed a similar quasi-combinatorial language for describing $T$-varieties of higher complexity,  starting with Altmann and Hausen's paper \cite{AH}.  Following this, there have been many papers studying the geometrical and combinatorial properties of $T$-varieties (see for instance \cite{AHS},\cite{AP},\cite{HS},\cite{IS},\cite{Topology},\cite{PS} and the references in \cite{TVAR}).

Here we will study algebraic cycles on T-varieties of complexity one. Our first main result is a description of the cones of effective cycles:

\begin{theorem} \label{theorem:rationalPolyhedral1}
The pseudoeffective cones $\overline{\Eff}_k(X)$ of a complete $T$-variety $X$ of complexity one are rational polyhedral, generated by classes of invariant subvarieties.
\end{theorem}

This generalizes Scott's results on the pseudoeffective cone of curves on a $T$-variety \cite{T-Curves}. There are not many examples where all pseudoeffective cones of cycles are known, this result gives a large class of examples where these are rational polyhedral. When $X$ is rational, it is known that $X$ is a Mori dream space \cite{HS}, so the statement was previously known for the cone of curves and effective divisors. However, the result applies also in the non-rational cases, showing that while the $T$-varieties $X$ are not Mori dream spaces, their cones of curves and divisors are still rational polyhedral. Moreover, even in the case $X$ is rational, and thus Mori dream, it is not a priori clear that the effective cones of cycles of intermediate dimensions should be rational polyhedral, as was shown by \cite[Example 6.10]{DELV}.

Our second main result gives a presentation for the Chow groups of $X$, in the case $X$ is a rational complete complexity one $T$-variety. The result is inspired by two different well-known results. First of all, for toric varieties, Fulton and Sturmfels showed that invariant subvarieties generate the Chow groups, and moreover they described the relations between these generators \cite[Proposition 2]{FS}. On the other hand, for a rational complete $T$-variety of complexity one, Altmann and Petersen give an analogous short exact sequence describing its Picard group \cite[Corollary 2.3]{AP}. These two results give some hints to how the Chow groups of a complexity one $T$-variety might look. Our second main result gives a complete description of these.

To explain the result, we first recall some basic facts about $T$-varieties. If $X$ is a complete rational $T$-variety of complexity one it comes equipped with a rational quotient map, that is, a rational map $f:X \dashrightarrow \p^1$, and another complete rational $T$-variety of complexity one, $\widetilde{X}$, with a $T$-equivariant blow-up map $r:\widetilde{X} \to X$ which resolves $f$; thus there is a map $g: \widetilde{X} \to \p^1$ such that $f \circ r = g$ whenever defined. Following \cite{IS}, we have that $X$ corresponds to  a fan $\Sigma$ giving the general toric fiber of $g$, as well as finitely many special fibers which correspond to polyhedral complexes with tailfan $\Sigma$. We denote the set of points in $\p^1$ such that the fiber is not general by $P$. In addition we need to keep track of the data of which cones corresponds to varieties contracted by $r$ (see Section \ref{section:preliminaries} for more details).

Letting the dimension of $X$ be $n+1$ and  $\Sigma$ be the fan describing the general fiber of $g$ (which is the toric variety $X_\Sigma)$ we define, for a non-negative integer $k$, the following sets:

$R_k =$ Cones of dimension $n+1-k$ corresponding to subvarieties not contracted by $r$.

$V_k =$ Faces of dimension $n-k$ of special fibers such that the tailcone corresponds to a subvariety not contracted by $r$.

$T_k =$ Cones of dimension $n-k$ corresponding to subvarieties contracted by $r$.

\begin{theorem}
For a complete rational T-variety $X$ of complexity one there is for any $0 \leq k \leq \dim X$ an exact sequence
\[  \bigoplus_{F \in V_{k+1}} M(F) \bigoplus_{\tau \in R_{k+1}} (M(\tau) \oplus \Z^P/\Z) \bigoplus_{\tau \in T_{k+1}} M(\tau) \to  \Z^{V_{k}} \oplus \Z^{R_{k}} \oplus \Z^{T_{k}} \to A_k(X) \to 0 \]
where for a rational polyhedral cone $\tau$ in $M_\Q, M(\tau) = \tau^\perp \cap M$ and $M(F)$ is the character lattice of the toric variety $F$ corresponds to.
\end{theorem}

The maps will be described below. For $k=n$ this coincides with the exact sequence of Altmann and Petersen. The above results also generalizes the results of Laface, Liendo and Moraga \cite{Topology}, where they give a presentation of the rational Chow ring of a complete complexity one $T$-variety which is contraction free, that is, when $X=\widetilde{X}$. Being contraction-free is however quite restrictive, for instance their rational Chow ring is generated by divisors \cite[Lemma 4.1]{Topology}, which is not true in general. We illustrate our results by studying examples such as toric downgrades (meaning we only remember the action of a codimension one torus on a toric variety), projectivizations of rank two toric vector bundles and the Grassmannian $\Gr(2,4)$, none of which (in general) are contraction free. \newline

\noindent
{\bf Acknowledgements.}
I am grateful to John Christian Ottem for suggesting I work on this problem and for numerous helpful conversations. Parts of this work was done while staying at IMPAN at the Simons semester ``Varieties: Arithmetic and Transformations'' September-October 2018 in Warsaw. In particular I wish to thank Nathan Ilten, Alvaro Liendo and Hendrik S\"{u}ss for answering many of my questions on $T$-varieties while I was there. I am grateful for helpful conversations with Corey Harris and Kathl\'en Kohn and for comments on a preliminary version of this paper by Ragni Piene and Hendrik S\"uss. This work was partially supported by the project Pure Mathematics in Norway, funded by Bergen Research Foundation and Troms{\o} Research Foundation and by the grant 346300 for IMPAN from the Simons Foundation and the matching 2015-2019 Polish MNiSW fund.

\section{Preliminaries on $T$-varieties} \label{section:preliminaries}

The papers \cite{AH} and \cite{AHS} give a general framework for describing $T$-varieties of any complexity, we briefly recall the set-up. Denote by $T \simeq (\C^\ast)^n$ a torus of dimension $n$ and let $M$ and $N$ denote the lattices of characters and one-parameter subgroups of $T$, respectively.

Recall that any polyhedron $\Delta$ can be decomposed as a Minkowski sum $\sigma+P$, where $\sigma$ is a unique polyhedral cone and $P$ is a polytope. Fixing a polyhedral cone $\sigma \subset N_\Q$, we consider the semigroup
\[ \Pol_\Q^+(N,\sigma) = \{ \Delta \subset N_\Q | \Delta \text{ is a polyhedron with tailcone }  \sigma \}\] 
We also allow $\emptyset$ as an element of $\Pol_\Q^+(N,\sigma)$. Let $Y$ be a normal and semiprojective variety and let $\CDiv(Y)$ denote the group of Cartier-divisors on $Y$. We consider  
\[ \D = \sum_Z \Delta_Z \otimes Z \in \Pol_\Q^+(N,\sigma) \otimes \CDiv(Y). \]
For $u \in \sigma^\vee \cap M$, we may consider the evaluation
\[ \D(u) = \sum_Z \min \langle \Delta_Z,u \rangle Z \in \CDiv_\Q(Y)\]
We call $\D$ a p-divisor on $(Y,N)$ if $\D(u)$ is semiample for all $u \in \sigma^\vee \cap M$, as well as big for $u$ in the interior of $\sigma^\vee \cap M$. To a p-divisor $\D$ we can associate the sheaf of rings $\OO_Y(\D) = \sum_{u \in \sigma^\vee \cap M} \OO_Y(\D(u))$. Then $X = \Spec \Gamma(Y,\OO_Y(\D))$ is an affine T-variety of complexity $\dim Y$. Also $\widetilde{X} = \Spec_Y \OO_Y(\D)$ is T-variety of complexity $\dim Y$ and there is an equivariant map $r: \widetilde{X} \to X$.

Altmann and Hausen \cite{AH} shows that  any affine T-variety arises from a p-divisor in this way.

If $\D$ and $\D'$ both are p-divisors on $(Y,N)$ we define their intersection $\D \cap \D'$ as having coefficient $\Delta_Z \cap \Delta_Z'$ on $Z$. We say that $\D \subset \D'$ if $\Delta_Z \subset \Delta_Z'$ for all $Z$. In that case we say that $\D$ is a face of $\D'$ if the induced map $X(\D) \to X(\D')$ is an open embedding; there is a technical condition \cite[Proposition 3.4]{AHS} for checking this.

\begin{definition}
A finite set $\mathcal{S}$ of p-divisors on $(Y,N)$ is called a divisorial fan if the intersection of any two p-divisors is a common face of both and $\mathcal{S}$ is closed under inersections.
\end{definition}

The condition on the intersections comes from the fact that one can glue the varieties $X(\D), X(\D')$ along the open set $X(\D \cap \D')$. Then Altmann, Hausen and S\"{u}ss \cite{AHS} show that any $T$-variety arises from a divisorial fan in this way. 

While the arguments in this paper are mostly geometric, the perspective of a divisorial fan will be useful. Also, in complexity one, there is a different perspective, due to Ilten and S\"{u}ss \cite{IS}, which has the advantage of avoiding the technical condition about the open embeddings $X(\D) \to X(\D')$. We will  alternate between these perspectives, depending on what is most convenient.

We now specialize to the case where $Y$ is a curve, denoted by $C$, which we can assume is a smooth curve \cite[Corollary 8.12]{AH}. Also we generally denote the divisor $Z$ by $p$, since divisors on curves correspond to sums of points. For a p-divisor $\D$ on $C$ we define its degree as $\deg \D = \sum_p \Delta_p$ where we by sum of polyhedra mean Minkowski sum.

\begin{definition} \label{definition:markedfansy}
A \emph{marked fansy divisor} on $C$ is a formal sum $\Xi = \sum S_p \otimes [p]$ together with a fan $\Sigma \subset N_\Q$ and a subset $K \subset \Sigma$ such that 
\begin{enumerate}
\item Each $S_p$ is a complete polyhedral subdivison of $N_\Q$ such that $\tail (S_p) = \Sigma$ for all $p$.
\item If $\sigma \in K$ has full dimension then  $D^\sigma = \sum D^\sigma_p \otimes [p]$ is a p-divisor, where $D^\sigma_p$ is the polyhedron in $S_p$ with $\tail(D^\sigma_p)=\sigma$. 
\item for a full dimensional cone $\sigma \in K$ and a face $\tau$ of $\sigma$, $\tau \in K$ if and only if $\deg D^\sigma \cap \tau \neq \emptyset$. \label{nr3}
\item If $\tau$ is a face of $\sigma$, then $\tau \in K$ implies that $\sigma \in K$.
\end{enumerate}
\end{definition}

The subsets $\deg D^\sigma$ glue to a subset $\deg \Xi \subset N_\Q$.

Then \cite{IS} shows that any complete T-variety of complexity one corresponds to a marked fansy divisor. In other words, in this case we only need to remember the polyhedral subdivisions of the fibers as well as the set $K$ which records which subvarieties are contracted by $r$.

Given a complete rational complexity one $T$-variety defined by a divisorial fan $\mathcal{S}$, one can associate a marked fansy divisor by taking the polyhedral subdivisions given by $\mathcal{S}$ and marking all cones $\sigma$ such that there exists $\D \in \mathcal{S}$ with tailcone $\sigma$ and such that no coefficients of $\D$ equals $\emptyset$. The variety $\widetilde{X}$ is given as a marked fansy divisor by the same subdivisions as for $X$ but with $K=\emptyset$.

The intuition here is the following: For any point $p \in Y$ the fiber of $g$ is given by the polyhedral subdivison at the point $p$. The fan $\Sigma$ defines a toric variety which is the general fiber, but at some points there might be other fibers. Nevertheless these fibers are all unions of toric varieties. The T-action on $\widetilde{X}$ restricted to a fiber of $g$ is just the T-action on the fiber as a union of toric varieties.

\begin{figure}  
\begin{subfigure}[t]{.33\linewidth}
\begin{tikzpicture} [scale=0.8]
\node[draw,circle,inner sep=1pt,fill] at (0,0) {};
\draw [-] (0,0) -- (1,0);
\draw [-] (0,0) -- (-2,2);
\draw [-] (1,0) -- (3,2);
\draw [-] (1,0) -- (3,0);
\draw [-] (1,0) -- (1,-2);
\draw [-] (0,0) -- (0,-2);
\draw [-] (0,0) -- (-2,0);
\draw [-] (0,0) -- (0,2);
\end{tikzpicture}
\subcaption*{Fiber over $0$}
\end{subfigure}%
\begin{subfigure}[t]{.33\linewidth}
\begin{tikzpicture} [scale=0.8]
\node[draw,circle,inner sep=1pt,fill] at (0,0) {};
\draw [-] (0,0) -- (2,0);
\draw [-] (0,1) -- (2,3);
\draw [-] (0,1) -- (0,3);
\draw [-] (0,1) -- (-2,3);
\draw [-] (0,0) -- (0,-2);
\draw [-] (0,0) -- (-2,0);
\draw [-] (0,0) -- (0,1);
\end{tikzpicture}
\subcaption*{Fiber over $1$}
\end{subfigure}%
\begin{subfigure}[t]{.33\linewidth}
\begin{tikzpicture} [scale=0.8]
\node[draw,circle,inner sep=1pt,fill] at (0,0) {};
\draw [-] (0,0) -- (2,0);
\draw [-] (0,0) -- (2,2);
\draw [-] (-1,0) -- (-3,2);
\draw [-] (-1,0) -- (-3,0);
\draw [-] (-1,0) -- (-1,-2);
\draw [-] (0,0) -- (0,-2);
\draw [-] (0,0) -- (-1,0);
\draw [-] (0,0) -- (0,2);
\end{tikzpicture}
\subcaption*{Fiber over $\infty$}
\end{subfigure}
\caption{Polyhedral subdivisons defining a $T$-variety  $X$ of dimension three. }
\label{figure:bundloeoverp1p1}
\end{figure}
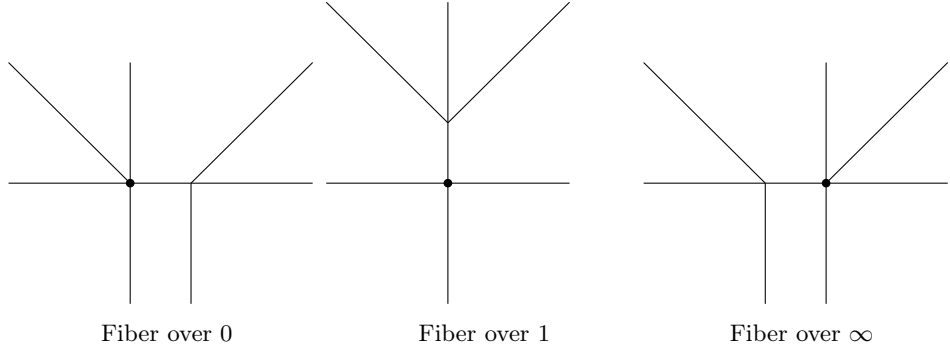
\begin{figure}  
\begin{tikzpicture} 
\draw [lightgray, fill=lightgray] (1,1) -- (-1,1) -- (-1,-2) -- (1,-2) -- (1,1) ;
\node[draw,circle,inner sep=1pt,fill] at (0,0) {};
\draw [dotted] (0,0) -- (0,-2);
\draw [-] (1,0) -- (3,0);
\draw [-] (1,1) -- (3,3);
\draw [-] (0,1) -- (0,3);
\draw [-] (-1,1) -- (-3,3);
\draw [-] (-1,0) -- (-3,0);
\draw [-] (-1,0) -- (-1,-2);
\draw [-] (1,0) -- (1,-2);
\draw [-] (1,1) -- (-1,1);
\draw [-] (1,1) -- (1,0);
\draw [-] (-1,1) -- (-1,0);
\end{tikzpicture}
\caption{ Everything except the gray area is $\deg \Xi$. }
\label{figure:degbundlep1p1}
\end{figure}
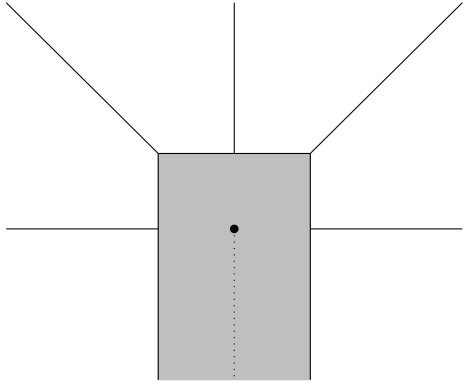

\begin{example} \label{example:bunldep1p1}
Figure \ref{figure:bundloeoverp1p1} shows polyhedral subdivisions for a complexity one $T$-variety $X$ of dimension three, with three special fibers. Its tailfan is the fan $\Sigma$ of the blow up of $\p^1 \times \p^1$ in two points contained in $\p^1 \times 0$. The subdivisions alone do not define $X$, we also need to specify the subset $K$. If we for instance choose that $K$ contains all maximal cones of $\Sigma$, we get  from Figure \ref{figure:degbundlep1p1} that all rays except the one generated by $(0,-1)$ are also in $K$.

This choice of $K$ corresponds to the projectivization of a rank two non-split toric vector bundle over $\p^1 \times \p^1$, see Section \ref{section:tvb}. 
\end{example}

\section{Subvarieties of T-varieties of complexity one} \label{section:subvarieties}

Given a fan $\Sigma$, we denote by $X_\Sigma$ the associated toric variety, by $B_\Sigma$ the torus-invariant boundary of $X_\Sigma$ and by $T_\Sigma$ the torus. For $\sigma \in \Sigma$, let $V(\sigma)$ denote the T-invariant subvariety of $X_\Sigma$ of codimension $\dim \sigma$ associated to $\sigma$. It is a toric variety with torus $T(\sigma)$, whose lattice of one-parameter subgroups is $N(\sigma) = N/\sigma \cap N$.

Fix a complete T-variety of complexity one $X$ with tailfan $\Sigma$. On the variety $\widetilde{X}$ there are various subvarieties arising from the combinatorial structure defining it. For a  point $p \in C$ and a face $F$ of the polyhedral complex $S_p$ defining the fiber $g^{-1}(y)$ there is the T-orbit $\orb(p,F)$ of dimension $\codim F$, the closure is denoted by $Z_{y,F}=\ol{\orb}(y,F)$.

We denote the generic point of $C$ by $\eta$. For any cone $\sigma \in \Sigma$ there is a T-invariant subvariety $B_\sigma$ which is given by the closure of $\orb(\eta,\rho)$, in other words $B_\sigma$ dominates $Y$ and in the general fiber is given by the subvariety $V(\sigma) \subset X_\Sigma$. We define $B= \cup_{\rho \in \Sigma(1)} B_\rho$. Then $B$ is a finite union of divisors on $\widetilde{X}$, each of which themselves is a T-variety of complexity one \cite[Proof of Proposition 4.12]{HS}. This fact will be important in the proof of Theorem \ref{theorem:rationalPolyhedral2}.

Fix any ray $\rho$ . Locally, $\widetilde{X}$ is given by a p-divisor $D = \sum_Z D_Z Z$ with tailcone $\sigma$. For a ray $\rho$, let $\pi:N_\Q \to N_\Q/ \Q \rho$ denote the projection sending $\rho$ to $0$. Then $D_\rho = \sum_Z \pi (D_Z) Z$ is a new p-divisor on $Y$ with tailcone $\pi(\sigma)$;  this defines $B_\rho$ as a T-variety. 

On $X$ some of the subvarieties $Z_{y,F}$ or $B_\sigma$ are contracted (those where $\tail(F)$ or $\sigma$ is in $K$), when they are not we also denote the corresponding subvarieties of $X$ by the same symbols. If $\sigma \in K$ is contracted then we denote the corresponding orbit closure $r(B_\sigma)$ in $X$ by $W_\sigma$, as well as denote $r(Z_{p,F})$ by $W_{p,F}$, for $\tail F = \sigma$.

We are interested in studying the pseduoeffective cones $\overline{\Eff}_k(X)$ inside the group $N_k(X)$ of cycles modulo numerical equivalence, which by definition is  the closure of the cone of effective $k$-cycles.

\begin{theorem} \label{theorem:rationalPolyhedral2}
The pseudoeffective cones $\overline{\Eff}_k(X)$ of a complete T-variety $X$ of complexity one are rational polyhedral, generated by classes of invariant subvarieties.
\end{theorem}
\begin{proof}
Let $X$ be a  T-variety of complexity one. Then the map $r$ induces a surjection 
\[ r_\ast: \overline{\Eff}_k(\widetilde{X}) \to \overline{\Eff}_k (X) \]
such that any extremal class in $\overline{\Eff}_k(X)$ is the image of an extremal class in $\overline{\Eff}_k(\widetilde{X})$ \cite[Theorem 1.4]{FL}. This implies that if all psuedoeffective cones of $\widetilde{X}$ are rational polyhedral, then so are all pseudoeffective cones of $X$. Thus we can and will assume that $X=\widetilde{X}$ and that the rational quotient map $f=g$ is an actual morphism.

Fix a subvariety $V$ of dimension $k$ of $X$ and choose a basis $v_1,...,v_n$ of the lattice $N$. Let $\lambda_i$ be the one-parameter subgroup of $T$ corresponding to $v_i$. We have a corresponding action $\C^\ast \times X \to X$ for each $\lambda_i$. Let $V_1$ be the flat limit as $t$ goes to zero of $\lambda_1 \cdot V$. Similarly let $V_i$ be the flat limit as t goes to zero of $\lambda_i \cdot V_{i-1}$. Then $V_n$ is an effective cycle numerically equivalent to $V$ and invariant under the entire torus action. Now for any irreducible component $W$ of $V_n$ we have two possibilities: Either $W$ dominates $C$ or it is contained in a fiber $X_y$.

If it is contained in a fiber $X_y$, then since it is irreducible it is contained in an irreducible component of the fiber. Each such component is itself a toric variety and there are only finitely many components where the fiber is different from $X_\Sigma$. If it is contained in a general fiber then by definition it is algebraically equivalent to a cycle in one fixed general fiber, in particular it is also numerically equivalent. Thus $W$ can be written as an effective sum of finitely many generators.

If it is not contained in a fiber, then we claim it has to be contained in the boundary divisor $B$. Indeed, if not, then there is a point $y \in C$ such that $X_y \simeq X_\Sigma$ and such that there is  $x$ in $(X_y \setminus B_y) \cap W_y=T_\Sigma \cap W_y$. Then $T_\Sigma$ acting on $x$ will be a set of dimension $n$ in the fiber $X_y$. But then $W_y$ has to have dimension at least $n$, since it is invariant under the entire torus action. However by assumption it is a proper closed subset of $X_\Sigma$ which has dimension $n$, so its dimension must be smaller than $n$.

Thus when $W$ is not contained in a fiber, it is contained in $B$ which is a finite union of  $T$-varieties of complexity one of smaller dimension than $\widetilde{X}$. A T-variety of dimension one and complexity one is simply a curve. The proposition is obviously true for any curve, thus by induction it is true for any T-variety of complexity one.

\end{proof}

In fact the above proof gives some information about the generators of the pseudoeffective cones of $X$:

\begin{corollary}
Let $X$ be a complete contraction-free T-variety of complexity one, in other words such that $\widetilde{X}=X$. Then the pseudoeffective cone $\overline{\Eff}_k(X)$ is generated by $B_\tau$, where $\dim \tau = n+1-k$ and $Z_{y,F}$, where $y \in Y$ and $\dim F =n-k$.
\end{corollary}
\begin{proof}
The proof of Theorem \ref{theorem:rationalPolyhedral2} implicity describes the generators: Any effective irreducible $k$-cycle $W$ has to either be contained in some $B_\rho$, $\rho \in \Sigma(1)$ or in some fiber $\widetilde{X}_y$.

If it is contained in a fiber $\widetilde{X}_y$ then, since the fiber is a union of irreducible toric varieties, it has to be contained in one of them, say the one corresponding to the vertex $v$ of $S_y$. After possibly replacing any compact face containing $v$ with the corresponding cone emanating from $v$ we have a complete fan $\Sigma'$ with vertex $v$ defining the toric variety. Now we know that as a cycle in $X_{\Sigma'}$, $W$ can be written as an effective sum $W \equiv \sum_{F_i \in \Sigma'(n-k)} a_i V(F_i)$. Letting $j_i: V(F_i) \to \widetilde{X}$ denote the inclusion, we have that $(j_i)_\ast V(F_i) = Z_{v,F_i}$. Thus in this case $W$ can be written as a positive sum of the finitely many $Z_{v,F}$.

Next we note the following: Since we are on $\widetilde{X}$ we know that no varieties are contracted by $r$, in particular any p-divisor must have $\emptyset$ as a coefficient. Thus any p-divisor for a $B_\rho$, which by the above is given by projecting the polyhedral coefficients of a p-divisor for $\widetilde{X}$, will also have an empty coefficient, in particular $B_\rho = \widetilde{B_\rho}$.

If an effective cycle is contained in some boundary divisor $B_\rho$ we argue by induction on $\codim W$. If $\codim W = 1$ then since $W$ is contained in $B_\rho$ and they are both irreducible of the same dimension, we must have $W=B_\rho$, hence we are done. If $\codim W >1$ then by induction $W$ can be written as a positive sum of $B_{p(\tau)}$ and $Z_{v,p(F)}$, where $p: N_\Q \to N_\Q/ \Q \rho$ is the projection. Denoting by $j: B_\rho \to \widetilde{X}$ the inclusion we have that $j_\ast (B_{p(\tau)}) = B_\tau$ and $j_\ast(Z_{v,p(F)}) = Z_{v,F}$. Thus we are done by induction.
\end{proof}

To study the generators of the pseudoeffective cones for general $X$ we need to also take into account the varieties that are contracted by $r$. First of all, it follows from \cite[Proof of Proposition 3.13]{PS} that for any invariant subvariety $Z$ of a T-variety $\widetilde{X}$ we have that
\[ \dim g(Z) + \dim r(Z) - \dim (Z) \geq 0\]
This combined with the fact that we always have $\dim(Z) \geq \dim(r(Z))$ we get that  if $\dim(g(Z)) = 0$ then $Z$ is not contracted. Moreover if $\dim(g(Z))=1$ and $Z$ is contracted then $\dim(r(Z))=\dim(Z)-1$. If $\sigma$ is a cone of dimension $k$, meaning that  $B_\sigma$ has codimension $k$ in $\widetilde{X}$, then if $\sigma \in K$, which means that $B_\sigma$ is contracted by $r$ to $W_\sigma$, by the above it is contracted to a subvariety of codimension $k+1$ in $X$. We will describe the images of the contracted varieties. \iffalse First we need some lemmas on stabilizer groups.

\begin{lemma}
Let $\sigma \in \Sigma$ and consider $B_\sigma \subset \widetilde{X}$ as a complexity one $T$-variety with respect to the torus $T(\sigma)$. Then for a face $Z_{q,F}$ with tailcone $\sigma$ we have that the generic isotropy group of the corresponding divisor in $B_\sigma$ is cyclic of order $\mu(v_F)$, where $v_F$ is the vertex which is the image of $F$ under the projection $N \to N/\Q \sigma$.
\end{lemma}
\begin{proof}
By the proof of \cite[Proposition 4.12]{HS} $B_\sigma$ is described as a $T$-variety by projecting the coefficients from $N$ to $N/\Q \sigma$. By \cite[Proposition 4.11]{HS} the order of the generic isotropy group is cyclic of order $\mu(v_F)$.
\end{proof}

\begin{lemma}
Assume $\sigma \in K$. Then the isotropy group of $W_\sigma \subset X$ with respect to $T(\sigma)$ is generated by the isotropy groups of all $Z_{p,F}$, with $\tail F = \sigma$.
\end{lemma}
\begin{proof}
For any $Z_{p,F}$ with $\tail F = \sigma$ we have that $r(Z_{p,F})=W_\sigma$. Since the map $r$ is equivariant we have that any element in the isotropy group of $Z_{p,F}$ is in the isotropy group of $W_\sigma$. Moreover if $t$ is not in the isotropy group of $Z_{p,F}$, then for any $x \in Z_{p,F}$ we have that $t \cdot x$ lies in some orbit $Z_{p,G}$ where $G$ does not contain $F$. But then $t \cdot r(x)$ cannot lie in the orbit $W_\sigma$ , hence $t$
\end{proof}

The following Lemma says exactly what the images of these contracted varieties are. \fi For a general point $q \in C$ we let $s_\sigma$ be the order of the group $[\Stab(Z_{q,\sigma}): \Stab(r(Z_{q,\sigma}))]$, of stabilizer groups.

\begin{lemma} \label{lemma:contractedSubvar}
Assume $\sigma \in \Sigma$ and that $\sigma \in K$. Then all faces $Z_{p,F}$ with $\tail F = \sigma$ are mapped to $W_\sigma$ under the contraction map $r$. Moreover we have the equality $ \mu(F) r_\ast([Z_{p,F}])=  s_\sigma [W_\sigma]$ of numerical classes, for $F$ with tailcone $\sigma$.
\end{lemma}
\begin{proof}

This  follows from \cite[Theorem 10.1]{AH}: Two faces $Z_{p,F}$ and $Z_{y,G}$ of the same polyhedral divisor $\D$ is identified under $r$ if  they have the same normal cone $\lambda$ and if $\D(u) = 0$ for some $u$ in the relative interior of $\lambda$ (if $C=\p^1$ this is in fact equivalent to $Z_{p,F}$ being identified with $Z_{y,G}$).

Let $\sigma \in K$ be a maximal cone. There is a polyhedral divisor $\D^\sigma$ with tailcone $\sigma$ and no empty coefficients. For any coefficient $\Delta_p^\sigma$ of $\D^\sigma$ the associated normal cone is the point $0$, moreover $\D(0)=0$, thus $Z_{p,\Delta_p^\sigma}$ is identified with $Z_{q,\Delta_q^\sigma}$ for any points $p,q \in C$, in particular, the entire  $B_\sigma$ in $\widetilde{X}$ is mapped to any fixed subvariety of the form $Z_{p,\Delta_p^\sigma}$.

If $\tau \in K$ is not maximal we have that $\cap_{\tau \preceq \sigma} \D^\sigma \cap \tau \neq \emptyset$ by Definition \ref{definition:markedfansy} (\ref{nr3}). This implies that in any fiber $S_p$ the intersection $\cap_{\tau \preceq \sigma} \Delta^\sigma_p \neq \emptyset$, which implies that in any fiber there is only one face $F_p$ with tailcone $\tau$. Let $\D^\tau$ be the p-divisor having the corresponding faces $F_p$ as coefficients. We have $F_p = \tau + Q_p$, where $Q_p$ is some polytope in $N_\Q$. Letting $v_p$ be a vertex of $Q_p$, we have that $F_p = \tau + v_p + (Q_p-v_p)$. Since there is only one face of $S_p$ with tailcone $\tau$, the polytope $Q_p-v_p$ must be contained in the linear span of $\tau$. This implies that all $F_p$ have the same normalcone $\lambda$. We wish to show that 
\[ \D^\tau(u) = \sum \min \langle F_p,u \rangle = 0 \]
for $u$ in the relative interior of $\lambda$. We have that $\min \langle F_p,u \rangle = \langle v_p,u \rangle$, since $u$ by definition is normal to any point in the linear span of $\tau$. Since $\deg \D^\tau = \deg \D^\sigma \cap \tau \subset \tau$ we must also have that $\sum v_p \in \tau$, thus $\langle \sum_p v_p,u \rangle=0$, which is what we wanted to show. Thus $Z_{p,F_p}$ is identified with $Z_{q,F_q}$ for any $p,q \in C$, thus the horizontal subvariety $B_\tau$ maps to any  $Z_{p,F_p}$.

For the last claim, we have that for a general fiber $q$ the map $r$ restricted to $Z_{q,\sigma}$, for a general point $q$, is finite of order $s_\sigma$, thus $r_\ast(Z_{q,\sigma}) = s_\sigma W_\sigma$. For a subvariety $Z_{p,F}$ of a special fiber the stabilizer group will have order $\frac{s_\sigma}{\mu(F)}$, thus $r_\ast (Z_{p,F}) = \frac{s_\sigma}{\mu(F)} W_\sigma$, which proves the claim.
\end{proof}

\begin{proposition} \label{proposition:effectiveconeX}
Let $X$ be a complete T-variety of complexity one. Then the pseudoeffective cone $\overline{\Eff}_k({X})$ is generated by $B_\tau$, where $\dim \tau = n+1-k$ and $\tau \notin K$, $Z_{y,F}$, where $y \in Y$ and $F \subset S_y$ has dimension $n-k$ and $\tail(F) \notin K$ and $W_\sigma$ where $\sigma$ has dimension $n-k$ and $\sigma \in K$.
\end{proposition}
\begin{proof}
As noted earlier $\overline{\Eff}_k(X)$ is the image of $\overline{\Eff}_k(\widetilde{X})$ via $r_\ast$, thus we know that $r(B_\tau)$ and $r(Z_{p,F})$ generate $\overline{\Eff}_k(X)$ as above. However we can omit $r(B_\tau)$ for $\tau \in \Sigma(n+1-k) \cap K$ since $r_\ast (B_\tau) = 0$ in this case. Fixing $\sigma \in K$ of dimension $n-k$ we have by the lemma above that all classes $r_\ast(Z_{p,F})$, for any point $p$ and $F$ with tailcone $\sigma$, are proportional, thus it is more convenient to only remember the single representative $W_\sigma$ instead of all the different $r(Z_{p,F})$.
\end{proof}

\section{Chow groups}

We now assume $X$ is a rational and complete T-variety of complexity one with tailfan $\Sigma$. Fix $k$ and, inspired by the above result, define the following sets

$R_k =$ Cones of dimension $n+1-k$ not contracted by $r$

$V_k =$ Faces of dimension $n-k$ of special fibers such that the tailcone is not contracted.

$T_k =$ Cones of dimension $n-k$ contracted by $r$

Then there is a surjection
\begin{equation} \Z^{R_k \cup V_k \cup T_k} \to A_k(X) \to 0 \label{surjection} \end{equation}
For $k=n$ this is the surjection for the Picard group of $X$ given in \cite{AP} (note that $T_n$ is always empty). Altmann and Petersen also describe the relations between the generators:
\[ 0 \to \Z^{P}/\Z \oplus M \to \Z^{V_n \cup R_n} \to \Picc(X) \to 0. \]
Recall that $P$ is the set of points $p$ in $\p^1$ where the polyhedral subdivison $S_p$ do not equal $\Sigma$, for convenience we assume $\infty \in P$ (we must assume that $P$ has at least size two, if it hasn't, simply add $1$ or $2$ points with fiber $\Sigma$). The first map is given by the following: generators for $\Z^{P}/\Z$ correspond to principal divisors $[p]-[\infty]$ on $\p^1$, one such generator is mapped to $\sum_{p,v} \mu(v) D_{p,v} - \sum_{\infty,v} \mu(v) D_{\infty,v}$. A character $m \in M$ is mapped to $\sum_v \mu(v) \langle m,v \rangle D_{p,v} + \sum_{R_n} \langle m,\rho \rangle D_\rho$. Here $\mu(v)$ is the smallest integer such that $\mu(v)v$ is a lattice point. 

\begin{theorem} \label{theorem:exactsequence}
For a rational complete complexity one T-variety $X$ there is for any $0 \leq k \leq \dim X$ an exact sequence
\[ \bigoplus_{F \in V_{k+1}} M(F) \bigoplus_{\tau \in R_{k+1}} (M(\tau) \oplus \Z^P/\Z) \bigoplus_{\tau \in T_{k+1}} M(\tau) \to  \Z^{V_{k}} \oplus \Z^{R_{k}} \oplus \Z^{T_{k}} \to A_k(X) \to 0. \]
\end{theorem}
The maps are given as follows:

If $F \in V_{k+1}$ then it corresponds to an invariant subvariety of an irreducible component of a fiber $p$ (possibly several components, if so pick one). Inside this component $F$ corresponds to a cone $\sigma_F$, thus $F$ corresponds to a toric subvariety of this component, we denote its character lattice by $M(F)$ (note that by \cite[Corollary 7.11]{AH} $M(F)$ could be a finite index sublattice of $M(\sigma_F)$) Then $m \in M(F)$ maps to $\sum_{\dim G = n-k, F \subset G} \langle m,v_{F,G} \rangle Z_{p,G}$, where $v_{F,G}$ generates $N(F)/N(G)$. This is the usual notion of rational equivalence on a toric variety. We note that if $\tail G \in K$, then $Z_{p,G}$ is not one of the generators of $V_k$, however by Lemma \ref{lemma:contractedSubvar} $ \mu(G) r_\ast(Z_{p,G})= s_\sigma W_{\tail G}$, thus we will identify $r_\ast(Z_{p,G})$ with the corresponding multiple of the generator $W_{\tail G}$  in $T_k$.

If $\tau \in R_{k+1}$ then $V(\tau)$ itself corresponds to a T-variety of complexity one. The map comes from this structure, as in the exact sequence of Altmann and Petersen. Explicitly a generator of $\Z^P/\Z$ corresponds to a divisor $[p]-[\infty]$ and it is mapped to $\sum_{F, \tail F = \tau} \mu(v_F) Z_{p,F} - \sum_{F, \tail F = \tau} \mu(v_F) Z_{\infty,F}$ where $v_F$ is the vertex corresponding to the image of $F$ in the divisorial fan corresponding to $V(\tau)$ as a T-variety (see Section \ref{section:subvarieties}). A character $m \in M(\tau)$ is mapped to $\sum_{F \in V_k, \tail F = \tau} \mu(v_F) \langle m,v_F \rangle Z_{p,F} + \sum_{\sigma \in R_k ,\tau \subset \sigma} \langle m,\bar{\sigma} \rangle Z_\sigma$ where $\bar{\sigma}$ is the ray which is the image of $\sigma$ in $N/N_\tau$.
 
If $\tau \in T_{k+1}$ then every cone $\sigma \in \Sigma$ containing $\tau$ will also be contracted by $r$, in particular any such $\sigma$ of dimension $n-k$ lies in $T_k$. Thus $m \in M(\tau)$ maps to $\sum_{\dim \sigma = n-k, \tau \subset \sigma} \langle m,v_{\tau,\sigma} \rangle r(B_\sigma)$. Again this is just the usual notion of rational equivalence on a toric variety.
 
\begin{proof}
As already noted Proposition \ref{proposition:effectiveconeX} implies that the final map is surjective. 

An invariant subvariety $W$ of $X$ corresponding to some face of a polyhedral subdivision could be non-normal. By \cite[Theorem 10.1]{AH} the corresponding subvariety of $\widetilde{X}$ is the normalization $\overline{W}$. Moreover by \cite[Proposition 5.2]{AH}  the torus of the normalization is the same as the torus of $W$ itself. By \cite[Ch. 5, Proposition 3.3]{GKZ} the normalization is in this case bijective. By \cite[Example 1.2.3]{Fulton} we thus have that the orders of vanishing of any rational function on $W$ is the same as the orders of vanishing on $\overline{W}$, in particular they can be computed using the usual techniques on normal toric varieties.

By construction of the maps giving the relations, we see that they are given by choosing an invariant subvariety $Z$ of dimension $k+1$ and choosing a rational function on $Z$ and taking its divisor. Thus by definition these will give relations in $A_k(X)$, thus the composition of the two maps are zero. 

By \cite[Theorem 1]{FMSS} the canonical homomorphism  $A_k^T(X) \to A_k(X)$ is an isomorphism, where $A_k^T(X)$ is the T-stable Chow group of $X$. By definition this is the quotient $Z_k^T(X)/R_k^T(X)$ of the group $Z_k^T(X)$ generated by T-invariant subvarieties of $X$ modulo the subgroup $R_k^T(X)$ generated by divisors of eigenfunctions on T-invariant $(k+1)$-dimensional subvarieties of $X$. In particular this implies that all relations in $A_k(X)$ come from divisors on T-invariant subvarieties.

In the exact sequence above we have by construction all such subvarieties and relations, except that we omit T-invariant subvarieties of general fibers such that the tailcone is not contracted. Letting $q$  be a point with general fiber we set  $P'= P \cup \{q\}$, thus we now consider $q$ as a special fiber. We show that every new subvariety and every new relation in the corresponding exact sequence is already generated by those in the exact sequence from $P$. Thus we may omit any general fiber, hence the sequence in the theorem is exact. This is similar to Altmann-Petersen's proof for the case of divisors \cite[Corollary 2.3]{AP}.

For any cone $\sigma \in \Sigma(n-k)$ where $\sigma \notin K$ we get the subvariety $Z_{q,\sigma}$ as a summand in $V_k$. There will also be a new relation coming from considering $\sigma \in R_{k+1}$ and observing that $\Z^{P'}/\Z$ has rank one more than $\Z^P/\Z$. The extra relation express $Z_{p,F}$ in terms of the other generators, thus we may omit $Z_{q,\sigma}$ as a generator in $V_k$, as well as the corresponding relation.

There is one more class of relations coming from adding $q$: Each cone $\tau$ of dimension $n-k-1$ with $\tau \in K$ gives an element in $V_{k+1}$. Any $m \in M(\tau)$ gives the relation 
\[ \sum_{\tau \preceq \sigma, \dim \sigma = n-k} \langle m,v_{\tau,\sigma} \rangle Z_{q,\sigma} \]
Now we already have relations saying that 
\[Z_{q,\sigma} = \sum_{H \preceq S_\infty, \tail H = \sigma} \mu(H) Z_{\infty,H}. \]
Thus we need to show that the relation
\[ \sum_{\tau \preceq \sigma, \dim \sigma = n-k} \langle m,v_{\tau,\sigma} \rangle  \sum_{H \preceq S_\infty, \tail H = \sigma} \mu(H) Z_{\infty,H}  \]
is in the group of relations generated by sequence in the theorem.

Pick a face $G$ of $S_\infty$ of dimension $n-k-1$ with tailcone $\tau$. Then $G = \tau + Q$, where $Q$ is a polytope. Letting $v$ be a vertex of $Q$ we have that $G=\tau +v + (Q-v)$. Since $\dim \tau = \dim G$ we must have that $(Q-v)$ is contained in the linear span of $\tau$. After taking the quotient $N/\Q \tau$ the image of $G$ will thus be the vertex $\bar{v}$, the image of $v$. Let $e_1,...,e_k$ be a $\Z$-basis for $\bar{v}^\perp$ and $e_0,e_1,...,e_k$ a $\Z$-basis for $M(\tau)$. We have that $M(G)$ is generated by $\mu(G)e_0,e_1,...,e_k$. This implies that the lattice index $[N(\tau):N(G)]$ equals $\mu(G)$. We then have the relation
\[ \sum_{G \preceq H, \dim H = n-k} \langle \mu(G)e_i,v_{G,H} \rangle Z_{\infty,H}. \]
For any $H$ with tailcone $\sigma$ we have that $\mu(G) v_{G,H} = \mu(H) v_{\tau,\sigma}$, by Lemma \ref{lemma:latticeIndex} (see below), thus the relation equals
\[ \sum_{G \preceq H, \tail H = \sigma} \mu(H) \langle e_i,v_{\tau,\sigma} \rangle Z_{\infty,H}  + \sum_{G \preceq H, \tail H = \tau} \mu(G) \langle e_i,v_{G,H} \rangle Z_{\infty,H}, \]
where all $H$ in the sums are of dimension $n-k$.
Now, if $H$ contains $G$, has dimension $n-k$ and tailcone $\tau$ then the image $\overline{H}$ in $N/\Q \tau$ is a compact edge $e$ having $\ol{G}$ has one of its vertices. The other vertex also corresponds to some $G'$ with tailcone $\tau$ and of dimension $n-k-1$. There is a similar relation to the above, corresponding to $G'$. Now $\mu(G)v_{G,H}$ equals a primitive generator in $N/\tau \cap N$ for the linear space spanned by $e$. Similarly $\mu(G')v_{G',H}$ equals a primitive generator the same linear space, but with different sign. Thus if we sum the relations from $G$ and $G'$ the term $Z_{\infty,H}$ will cancel. Thus if we sum all the relations corresponding to all possible such $G$s once, we see that the resulting relation is 
\[  \sum_G \sum_{G \preceq H, \dim H = n-k, \tail H = \sigma} \mu(H) \langle e_i,v_{\tau,\sigma} \rangle Z_{\infty,H} \]
By grouping together terms corresponding to the same cone $\sigma$ we see we can write this relation as 
\[ \sum_{\tau \preceq \sigma, \dim \sigma = n-k} \langle e_i,v_{\tau,\sigma} \rangle  \sum_{H \preceq S_\infty, \tail H = \sigma} \mu(H) Z_{\infty,H}  \]
which is the relation we wanted to show for $m=e_i$. Since this is true for any $i$ it will also follow for any $m$ by linearity.
\end{proof}

\begin{lemma} \label{lemma:latticeIndex}
Assume $\tau \preceq \sigma$ are cones satisfying $\dim \tau +1 = \dim \sigma$. Assume $G \preceq H$ are faces of some $S_p$, with $\tail G = \tau, \tail H =\sigma$, $\dim G = \dim \tau$ and $\dim H = \dim \sigma$. Then $\mu(G) v_{G,H} = \mu(H) v_{\tau,\sigma}$.
\end{lemma}
\begin{proof}
We may assume $e_1,...,e_k$ is a basis for $N(H)$ and $e_1,...,e_{k+1}$ is a basis for $N(G)$. $N(\tau)$ is a sublattice of $N(H)$ and there is an upper triangular integer matrix $B$ such that $\{ b_i = B e_i \}$ is a basis for $N(\tau)$. In particular the index $\mu(G)=[N(G):N(\tau)]$ equals the product of the diagonal entries $\beta_i$ of $B$. We may also assume that $b_1,...,b_k$ is a basis for $N(\sigma)$. By definition $v_{\tau,\sigma}$ is a generator of $N(\tau)/N(\sigma)$, in our chosen basis we can choose it as the image of $b_{k+1}$ in the quotient. Similarly the image of $e_{k+1}$ is a generator of $N(G)/N(H)$. Thus we see that $\beta_{k+1} v_{G,H}=v_{\tau,\sigma}$. We also have that $\mu(H)=[N(H):N(\sigma)] = \beta_1 \cdots \beta_k$, in particular $\beta_{k+1}= \frac{\mu(G)}{\mu(H}$, which proves the statement.
\end{proof}

\begin{remark}
We denote the number of elements of $R_k,V_k,T_k$ by $r_k,v_k,t_k$, respectively.
\end{remark}

\begin{example}
Let $X=\Gr(2,4)$ which we can identify with the quadric \[V(p_{12}p_{34}-p_{13}p_{24}+p_{14}p_{23}) \subset \p^5. \]
The group $SL_4(\C)$ acts on $\Gr(2,4)$ and contains the subgroup of diagonal matrices which is a $3$-dimensional torus and which acts effectively on $\Gr(2,4)$ making it into a $T$-variety of complexity one. The rational quotient $Y$ is naturally identified with the moduli space of marked genus $0$ curves $\ol{M_{0,4}} \simeq \p^1$ (see \cite{Kapranov}). The identification is defined as follows. We now think of $\Gr(2,4)$ as the space of lines in $\p^3=\p(V)$ where $V$ has basis $x_1,...,x_4$. If $l \in \Gr(2,4)$ is a general line then the intersections $l \cap \{ x_i=0 \}$ will give four points $p_1,p_2,p_3,p_4$ on $l$. Then the cross ratio $\mathrm{CR}(p_1,p_2,p_3,p_4)$ defines the rational map to $\p^1$. In this way we obtain all values in $\p^1$ except $\{0,1,\infty \}$. The points $0,1,\infty$ are obtained by the non-general lines in the sets $V(p_{14}) \cup V(p_{23}), V(p_{12}) \cup V(p_{34}), V(p_{13}) \cup V(p_{24})$, respectively.

In coordinates the map is given as follows: On the open affine set $D(p_{12})$ a point in Plucker coordiates maps to $(p_{13}p_{24} : p_{23}p_{14} ) \in \p^1$. The indeterminancy locus is given when both coordinates equals $0$, we see that this locus consists of the eight planes
\[ Z_k^+= \{ V(p_{ij}) | k \in \{i,j \} \} \]
\[Z_k^-= \{ V(p_{ij}) | k \notin \{i,j \} \}. \]

We can resolve the quotient map by blowing up the union of the eight planes to get a map $\widetilde{X}=\Bl \Gr(2,4) \to \p^1$. We will reinterpret this example in the language of $T$-varieties.

The paper \cite{AltH} exhibits a divisorial fan for $\Gr(2,4)$: Let $N=\Z^4/\Z$ and let $e_1,e_2,e_3,e_0$ denote the image of the standard basis vectors of $\Z^4$, thus $e_0 = -e_1-e_2-e_3$. The tailfan $\Sigma$ is the toric threefold with maximal cones $\Cone(\pm e_1, \pm e_2, \pm e_3 \pm e_4 | \text{there are exactly 2 pluses and 2 minuses})$. This has $6$ maximal cones, $12$ cones of dimension two and $8$ rays. The special fibers correspond to the boundary divisors $\ol{M_{0,4}} \setminus M_{0,4}$ of reducible genus $0$ curves, of which there are three. They correspond to partitions $($(\{1,4\}, \{2,3\})$,(\{1,2\}, \{3,4 \})$, $(\{1,3\}, \{2,4 \})$, we may assume these correspond to the points  $0,1,\infty$, respectively.

The fiber over $0$ corresponds to replacing the origin with the compact edge $f_{23}$ with vertices $ (0,0,0)$ and $(-1,-1,0)$. Similarly in the fiber over $1$ we insert the edge $f_{12}$ with vertices $ (0,0,0)$  and $(-1,0,-1)$ and over $\infty$ the edge $f_{13}$ with vertices $(1,1,1)$ and $(1,0,0)$ (the polyhedra which is written in \cite[Theorem 4.2]{AltH} is a shifted version of the above, with rational coefficients. By \cite[p.8 Remark 2]{AltH} the true p-divisor correspond to a shift turning all polyhedra into lattice polyhedra, which is what we have done.) The p-divisor containing the edge $f_{ij}$ has the empty coefficient over the other two special fibers. For a cone $\sigma$ of $\Sigma$ the faces of the special fibers with tailcone $\sigma$ all belong to the same p-divisor. See Figure \ref{figure:g24}. 
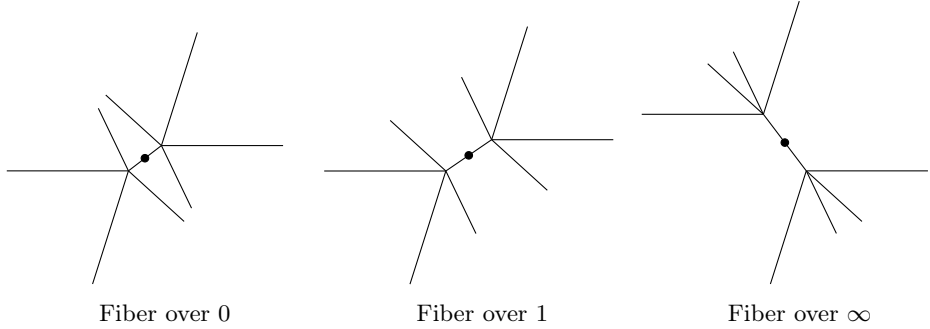
\begin{figure}  
\begin{subfigure}[t]{.33\linewidth}
\begin{tikzpicture} [scale=0.8]
\begin{scope}[rotate around x=40]
\node[draw,circle,inner sep=1pt,fill] at (0,0) {};
\draw [-] (0,0,0) -- (0,1/2,1/2);
\draw [-] (0,0,0) -- (0,-1/2,-1/2);
\draw [-] (0,5/2,1/2) -- (0,1/2,1/2);
\draw [-] (0,-5/2,-1/2) -- (0,-1/2,-1/2);
\draw [-] (0,1/2,5/2) -- (0,1/2,1/2);
\draw [-] (0,-1/2,-5/2) -- (0,-1/2,-1/2);
\draw [-] (-2,1/2,1/2) -- (0,1/2,1/2);
\draw [-] (2,-1/2,-1/2) -- (0,-1/2,-1/2);
\draw [-] (2,5/2,5/2) -- (0,1/2,1/2);
\draw [-] (-2,-5/2,-5/2) -- (0,-1/2,-1/2);

\end{scope}
\end{tikzpicture}
\subcaption*{Fiber over $0$}
\end{subfigure}%
\begin{subfigure}[t]{.33\linewidth}
\begin{tikzpicture} [scale=0.8]
\begin{scope}[rotate around x=40]
\node[draw,circle,inner sep=1pt,fill] at (0,0,0) {};
\draw [-] (0,0,0) -- (1/2,1/2,0);
\draw [-] (0,0,0) -- (-1/2,-1/2,0);
\draw [-] (5/2,1/2,0) -- (1/2,1/2,0);
\draw [-] (-5/2,-1/2,0) -- (-1/2,-1/2,0);
\draw [-] (1/2,5/2,0) -- (1/2,1/2,0);
\draw [-] (-1/2,-5/2,0) -- (-1/2,-1/2,0);
\draw [-] (1/2,1/2,-2) -- (1/2,1/2,0);
\draw [-] (-1/2,-1/2,2) -- (-1/2,-1/2,0);
\draw [-] (5/2,5/2,2) -- (1/2,1/2,0);
\draw [-] (-5/2,-5/2,-2) -- (-1/2,-1/2,0);

\end{scope}
\end{tikzpicture}
\subcaption*{Fiber over $1$}
\end{subfigure}%
\begin{subfigure}[t]{.33\linewidth}
\begin{tikzpicture} [scale=0.8]
\begin{scope}[rotate around x=40]
\node[draw,circle,inner sep=1pt,fill] at (0,0) {};
\draw [-] (0,0,0) -- (1/2,0,1/2);
\draw [-] (0,0,0) -- (-1/2,0,-1/2);
\draw [-] (5/2,0,1/2) -- (1/2,0,1/2);
\draw [-] (-5/2,0,-1/2) -- (-1/2,0,-1/2);
\draw [-] (1/2,0,5/2) -- (1/2,0,1/2);
\draw [-] (-1/2,0,-5/2) -- (-1/2,0,-1/2);
\draw [-] (1/2,-2,1/2) -- (1/2,0,1/2);
\draw [-] (-1/2,2,-1/2) -- (-1/2,0,-1/2);
\draw [-] (5/2,2,5/2) -- (1/2,0,1/2);
\draw [-] (-5/2,-2,-5/2) -- (-1/2,0,-1/2);
\end{scope}
\end{tikzpicture}
\subcaption*{Fiber over $\infty$}
\end{subfigure}
\caption{Polyhedral subdivisons defining $\Gr(2,4)$. From each of the non-zero vertices of the polyhedral subdivisions there is emanating a three-dimensional cone with four generators.}
\label{figure:g24}
\end{figure}
Thus we get the following numbers:
\[ r_3=0,  v_3=6, t_3=0 \]
\[ r_2=0,  v_2=3, t_2=8 \]
\[ r_1=0, v_1=0, t_1=12 \]
\[ r_0=0, v_0=0, t_0=6. \]
The six invariant subvarieties $V(p_{ij})$ correspond to $V_3$.  $V_2$ corresponds to the subvarieties $V(p_{12},p_{34})$, up to permutation, while $t_2=8$ says exactly that there are eight invariant subvarieties of codimension two which are blown up by $r$, they correspond to the sets $Z_k^\pm$.
Thus we get an exact sequence 
\[ \Z^{18} \to \Z^{11} \to A_2(\Gr(2,4)) \to 0 \]
The $\Z^{11}$ has generators $W_{12},W_{13},W_{23}$ corresponding to the edges $f_{ij}$, and $E_i^\pm$ corresponding to the rays $\pm e_i$. The $\Z^{18}$ corresponds to six copies of $\Z^3$, one for each vertex of a special fiber. Fixing for example the vertex $v=(0,0,0)$ of $f_{13}$, we have that this is the vertex of a toric variety with rays with directions
\[ \begin{vmatrix} 1 \\ 0 \\ 0 \end{vmatrix},\begin{vmatrix} 0 \\ 1 \\ 0 \end{vmatrix},\begin{vmatrix} 0 \\ 0 \\ -1 \end{vmatrix},\begin{vmatrix} 1 \\ 1 \\ 1 \end{vmatrix},\begin{vmatrix} -1 \\ -1 \\ 0 \end{vmatrix}. \]
The relations we get from this will be 
\[ \ddiv(1,0,0) = E_1^+ + E_0^- - W_{13} \]
\[ \ddiv(0,1,0) = E_2^+ + E_0^- - W_{13} \]
\[ \ddiv(0,0,1) = E_3^- - E_0^-.\]
Doing this for all vertices we see that all $E_i^-$ are identified, call this class $E^-$, similarly all $E_i^+$ are identified, call this $E^+$ and all $W_{ij}$ are identified, call this $W$. Then this gives a presentation 
\[ A_2(Gr(2,4)) \simeq \Z(E^+,E^-,W)/(E^+ + E^- - W)\] 
which we see is isomorphic to a wellknown presentation of this group, namely \[ \Z(s_{1,1},s_2, s_1^2) / (s_{1,1} + s_2 - s_1^2)\]
Here the $s_1,s_2,s_{1,1}$ correspond to Schubert cycles in the Chow ring of $\Gr(2,4)$.

Similarly there is an exact sequence
\[ \Z^{22} \to \Z^{12} \to A_1(Gr(2,4)) \to 0 \]
The generators of $\Z^{12}$ correspond to the $12$ two-dimensional faces of $\Sigma$, they are of the form $\Cone(e_i,-e_j)$. Denote the corresponding generator by $Z_{i,-j}$. The relations come from two types: three copies of $\Z^2$ coming from the edges $f_{ij}$. Writing out the relations we see that they give $Z_{i,-j} = Z_{j,-i}$. Also there are eight copies of $\Z^2$ corresponding to the rays of $\Sigma$, they give relations $Z_{i,-j} = Z_{x,-k}$. Combining these we see that all $Z_{i,-j}$ are identified, thus $A_1(Gr(2,4))$ is one-dimensional, as expected.
\end{example}

\section{Toric downgrades}

In this section we study the example of downgrading a toric variety to only consider it as a T-variety of complexity one. Not surprisingly the exact sequence of Theorem  \ref{theorem:exactsequence} coincides with the exact sequence of Fulton-Sturmfels.

We now conisder a toric variety coming from a fan $\Sigma$ living in $\Z^{n+1} \otimes \Q$. We choose a splitting $\Z^{n+1} = \Z^n \oplus \Z = N \oplus \Z$ and consider $X_\Sigma$ only with the action of $T_N$. We then have an exact sequence
\[ 0 \to N \to N \oplus \Z \to \Z \to 0 \]
The $\Z$ corresponds to the quotient, which for us will be $\p^1$. By construction there will only be two special fibers, over $0$ and $\infty$. We denote by $s$ the section $N \oplus \Z \to N$ and by $\phi$ the map $N \oplus \Z \to \Z$. For a cone $\sigma \in \Sigma$ we get a polyhedral divisor with tailcone $\sigma \cap N$ and coefficient $s(\sigma \cap \phi^{-1}(1))$ over $[0]$ and coefficient $s(\sigma \cap \phi^{-1}(-1))$ over $[\infty]$.

We consider the vector space $V=\Q^{n+1} = \Z^{n+1} \otimes \Q$ with basis $v_1,...,v_{n+1}$ and denote the last coordinate hyperplane by $H= \{ v = \sum t_iv_i | t_{n+1}=0 \}$ and $H_{\geq 0}= \{ v = \sum t_iv_i | t_{n+1} \geq 0 \}$, $H_{> 0}= \{ v = \sum t_iv_i | t_{n+1} > 0 \}$ and similarly for $H_{\leq 0}, H_{<0}$.

\begin{lemma} \label{lemma:downgradecycle}
A cone $\sigma \in \Sigma$ of dimension $n-k+1$ corresponds to 
\begin{enumerate}
    \item an element of $R_k$ if and only if $\sigma \subset H$.
    \item an element of $V_k$ if and only if $\sigma \subset H_{\leq 0}$ or $\sigma \subset H_{\geq 0}$, but $\sigma$ is not contianed in $H$.
    \item an element of $T_k$ if and only if $\sigma$ intersects both $H_{>0}$ and $H_{<0}$.
\end{enumerate}
\end{lemma}
\begin{proof}
If $\sigma \subset H$ then wee see that the tailcone of the associated polyhedral divisor is $\sigma$. Moreover we see that this divisor will have $\emptyset$ as coefficient over $[0]$ and $[\infty]$, thus it will not be an element of K.

If $\sigma \subset H_{\leq 0}$ then the coefficient of $[0]$ will be empty thus the tailcone will not lie in $K$. Moreover the coefficient over $[\infty]$ will have dimension $n-k$ as it will equal a compact polyhedron with vertices corresponding to generators having strictly positive last coordiante plus the tailcone which corresponds to rays with zero last coordinate.

If $\sigma$ intersects both $H_{>0}$ and $H_{<0}$, then first of all we see that there will be no $\emptyset$ coefficients. Moreover the tailcone will be the intersection of $\sigma$ with $H$, in particular it will have dimension $n-k$.

Since any cone $\sigma$ belongs to only one of the three categories the only if statements follow as well.
\end{proof}

\begin{corollary} \label{corollary:downgrade}
For a downgraded toric variety the exact sequence of Theorem \ref {theorem:exactsequence} equals the exact sequence
\[   \oplus_{\sigma \in \Sigma(n-k)} M(\sigma) \to   \oplus_{\sigma \in \Sigma(n+1-k)} \Z  \to A_k(X) \to 0 \]
from \cite{FS}.
\end{corollary}
\begin{proof}
This follows from the above lemma, together with the fact that that for $\tau \in R_k$ the relation coming from $\Z^P/\Z$, when the number of special fibers is two, is the same as the relation coming from the last factor of the torus corresponding to $N \oplus \Z$, since both are saying that the fibers over $[0]$ and $[\infty]$ of the corresponding $\p^1$ are equivalent.
\end{proof}

In particular, we see from the above that the number of elements of individual $R_k,V_k,T_k$ can vary a lot, depending on which subtorus we choose, even if their sum is constant, equal to the number of cones in $\Sigma$ of dimension $n-k+1$.

\begin{example} \label{example}
Let the fan for $\p^2$ be given from rays $\rho_1 = (1,0), \rho_2 = (0,1), \rho_0= (-1,-1)$ and denote the associated divisors by $D_i$. Let $\E= D_1 \oplus 0$ and $\F = (D_1+D_2) \oplus D_0$. Then $X=\p(\E) \simeq \p(\F)$ and it is a toric variety, however the different choices for $\E$ and $\F$ corresponds to different $T$-structures on $X$, giving different polyhedral subdivisions and marked cones defining it as a $T$-variety, see Figures \ref{figure:bundloeoverp21} and \ref{figure:bundloeoverp22}.

For $\p(\E)$ the only marked cones are the ones containing the ray $(1,0)$, while for $\p(\F)$ all non-zero cones are marked. This is both an example of a toric downgrade, as well as an example of a toric vector bundle (see the next section).
\end{example}

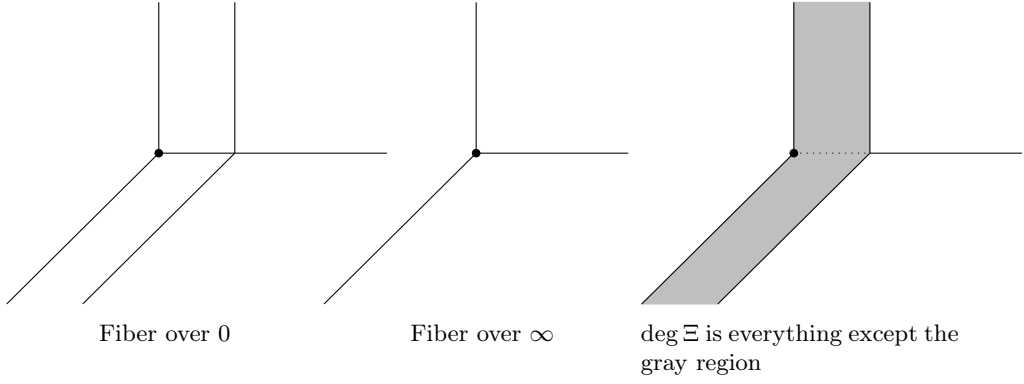
\begin{figure} 
\begin{subfigure}[t]{.33\linewidth}
\begin{tikzpicture} 
\node[draw,circle,inner sep=1pt,fill] at (0,0) {};
\draw [-] (0,0) -- (1,0);
\draw [-] (1,0) -- (3,0);
\draw [-] (1,0) -- (1,2);
\draw [-] (0,0) -- (0,2);
\draw [-] (0,0) -- (-2,-2);
\draw [-] (1,0) -- (-1,-2);
\end{tikzpicture}
\subcaption*{Fiber over $0$}
\end{subfigure}%
\begin{subfigure}[t]{.33\linewidth}
\begin{tikzpicture}
\node[draw,circle,inner sep=1pt,fill] at (0,0) {};
\draw [-] (0,0) -- (1,0);
\draw [-] (1,0) -- (2,0);
\draw [-] (0,0) -- (0,2);
\draw [-] (0,0) -- (-2,-2);
\end{tikzpicture}
\subcaption*{Fiber over $\infty$}
\end{subfigure}%
\begin{subfigure}[t]{.33\linewidth}
\begin{tikzpicture}
\draw [lightgray, fill=lightgray] (1,0) -- (-1,-2) -- (-2,-2) -- (0,0) -- (1,0);
\draw [lightgray, fill=lightgray] (1,0) -- (1,2) -- (0,2) -- (0,0) -- (1,0);
\node[draw,circle,inner sep=1pt,fill] at (0,0) {};
\draw [dotted] (0,0) -- (1,0);
\draw [-] (1,0) -- (3,0);
\draw [-] (1,0) -- (1,2);
\draw [-] (0,0) -- (0,2);
\draw [-] (0,0) -- (-2,-2);
\draw [-] (1,0) -- (-1,-2);
\end{tikzpicture}
\subcaption*{$\deg \Xi$ is everything except the gray region}
\end{subfigure}
\caption{Let $\E = D_1 \oplus 0$ on $\p^2$. The associated $T$-variety only has one special fiber over $0$.}
 \label{figure:bundloeoverp21}
\end{figure}

\begin{figure}
\begin{subfigure}[t]{.33\linewidth}
\begin{tikzpicture} [scale=0.8]
\node[draw,circle,inner sep=1pt,fill] at (0,0) {};
\draw [-] (0,0) -- (1,0);
\draw [-] (1,0) -- (3,0);
\draw [-] (1,0) -- (0,1);
\draw [-] (0,0) -- (0,1);
\draw [-] (0,1) -- (0,3);
\draw [-] (0,1) -- (-2,3);
\draw [-] (1,0) -- (3,-2);
\draw [-] (0,0) -- (-2,-2);
\end{tikzpicture}
\subcaption*{Fiber over $0$}
\end{subfigure}%
\begin{subfigure}[t]{.33\linewidth}
\begin{tikzpicture} [scale=0.8]
\node[draw,circle,inner sep=1pt,fill] at (0,0) {};
\draw [-] (0,0) -- (2,0);
\draw [-] (0,0) -- (0,2);
\draw [-] (-1,-1) -- (-3,1);
\draw [-] (-1,-1) -- (1,-3);
\draw [-] (0,0) -- (-1,-1);
\draw [-] (-1,-1) -- (-3,-3);
\end{tikzpicture}
\subcaption*{Fiber over $\infty$}
\end{subfigure}%
\begin{subfigure}[t]{.33\linewidth}
\begin{tikzpicture} [scale=0.8]
\draw [lightgray, fill=lightgray] (1,0) -- (0,1) -- (-1,0) -- (-1,-1) -- (0,-1) -- (1,0);
\node[draw,circle,inner sep=1pt,fill] at (0,0) {};
\draw [-] (1,0) -- (0,1);
\draw [-] (0,1) -- (-1,0);
\draw [-] (-1,0) -- (-1,-1);
\draw [-] (-1,-1) -- (0,-1);
\draw [-] (0,-1) -- (1,0);
\draw [-] (1,0) -- (3,0);
\draw [-] (0,-1) -- (2,-3);
\draw [-] (0,1) -- (0,3);
\draw [-] (-1,0) -- (-3,2);
\draw [-] (-1,-1) -- (-3,-3);

\draw [dotted] (0,0) -- (1,0);
\draw [dotted] (0,0) -- (-1,0);
\draw [dotted] (0,0) -- (0,1);
\draw [dotted] (0,0) -- (0,-1);

\end{tikzpicture}
\subcaption*{$\deg \Xi$ is everything except the gray region}
\end{subfigure}
\caption{Let $\F = (D_1+D_2) \oplus D_0$ on $\p^2$. Then $\p(\F)$ is isomorphic to $\p(\E)$, however the different choice of $T$-structure gives different data as $T$-varieties.}
 \label{figure:bundloeoverp22}
\end{figure}
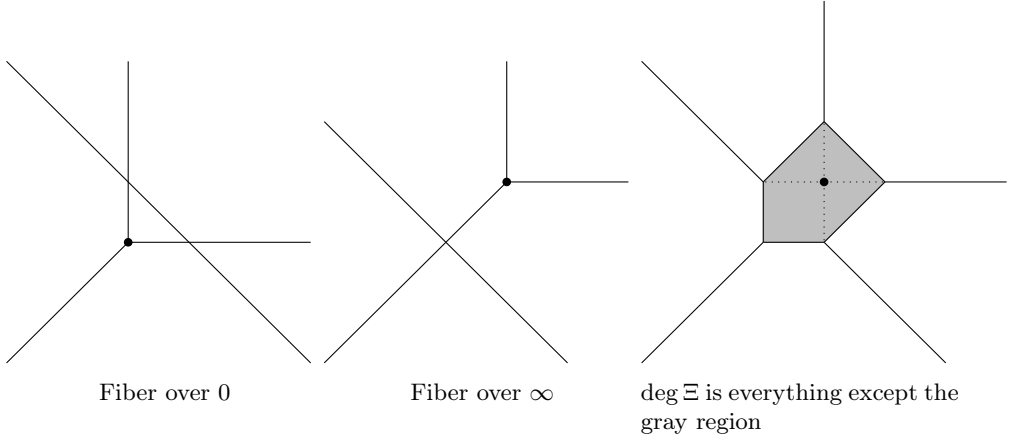

\section{Toric Vector bundles as T-varieties} \label{section:tvb}

A well-studied example of T-varieties are toric vector bundles. Fix a toric variety $X_\Sigma$. A vector bundle $\E$ on $X_\Sigma$ is called toric if there is a T-action on the geometric vector bundle which is compatible with the action on $X_\Sigma$ and linear on the fibers of $\E$. Toric vector bundles where classified by Klyachko \cite{Klyachko}: they correspond to, for each ray, a filtration of the fiber $E$ over the identity of the torus, indexed by integers.

Given an indecomposable toric vector bundle of rank $r+1$, the projectivization $\p(\E)$ can be considered a T-variety of complexity $r$ (if it decomposes then the complexity is lower). If $\E$ splits as a sum of line bundles then $\p(\E)$ is in fact a toric variety and can thus be described as a complexity one $T$-variety via downgrading. 

Fix now a smooth toric variety $X_\Sigma$, a toric vector bundle $\E$ on $X_\Sigma$ and fix a maximal cone $\sigma$ in $\Sigma$. Any vector bundle on an affine toric variety splits as a sum of line bundles which again implies that $\p(\E|_{U_\sigma})$ is a toric variety \cite[p.31]{Oda}. We can describe $\p(\E|_{U_\sigma})$ as a T-variety via downgrading the torus action to only remember the action on the base. The description of $\p(\E)$ as a T-variety will locally be glued from  such pieces.

This description of rank two toric vector bundles as T-varieties are given in \cite[Proposition 8.4]{AHS}. Fix a maximal cone $\sigma$. Then $\E|_{U_\sigma} = \OO(u_1) \oplus \OO(u_2)$, where $u_1,u_2$ are characters of the torus. Define the four polyhedra
\[ \Delta_1 = \{ v \in N_\Q | \langle u_1-u_2,v \rangle \geq 1 \} \cap \sigma \]
\[ \Delta_2 = \{ v \in N_\Q | \langle u_2-u_1,v \rangle \geq 1 \} \cap \sigma \]
\[ \nabla_1 = \{ v \in N_\Q | \langle u_1-u_2,v \rangle \leq 1 \} \cap \sigma \]
\[ \nabla_2 = \{ v \in N_\Q | \langle u_2-u_1,v \rangle \leq 1 \} \cap \sigma \]
Fix $v_1,v_2 \in E$ corresponding to the $1$-dimensional vector spaces in the Klyachko-filtration for $\E$ on $\sigma$. If there are no two distinct such spaces, simply choose any vectors such that $v_1,v_2$ is a basis, preferably $v_i$ whose span appears in filtration for other rays (this minimizes the number of special fibers). Then 
\[ \D_1= \Delta_1 \otimes v_1 + \nabla_2 \otimes v_2 \]
\[ \D_2= \nabla_1 \otimes v_1 + \Delta_2 \otimes v_2 \]
are polyhedral divisors describing $\p(\E|_{U_\sigma})$. Here $v_i$ is considered a point in $\p^1=\p(E)$. Thus the special fibers $P$ correspond to the distinct one-dimensional subspaces appearing in the Klyachko filtrations.

The description above enables us to write out which cycles get contracted by the map $r: \widetilde{X} \to X$. For a cone $\sigma$ there are essentially three different cases, corresponding to whether there are no one-dimensional linear spaces in the filtrations on the rays of $\sigma$, or only one distinct one-dimensional linear space, or if there are two different one-dimensional linear spaces (there cannot be more, by the compatibility condition for a toric vector bundle).

If $u_1=u_2$ (meaning that there are no one-dimensional spaces) then $\sigma$ itself is a cone of the tailfan and both polyhedral divisors have $\emptyset$ as a coefficient, hence $\sigma$ is non-contracted, hence no face of $\sigma$ will be contracted.

If $u_1 \geq u_2$, but $u_1 \neq u_2$ on $\sigma$ (meaning there is one distinct one-dimensional space) we have that $\D_1$ has tailcone $\sigma$ and no empty coefficients, thus $\sigma$ will be contracted. We see that $\Delta_2$ is $\emptyset$. Let $u_1 = (x_1,...,x_n)$ and $u_2=(y_1,...,y_n)$ and assume $x_i-y_i>0$ for $i=1,...,s$ and $=0$ for $i=s+1,...,n$. Then 
\[ \nabla_1 = \{ (a_1,...,a_n) \in N_\Q | a_i \geq 0, \sum_{i=0}^s a_i(x_i-y_i)  \leq 1 \} \]
We see that $\nabla_1=\tau+P$ where $\tau$ is the $n-s$-dimensional cone generated by $e_{s+1},...,e_n$ (where $\sigma = \Cone(e_1,...e_n)$) and $P$ is the $s$-dimensional simplex defined by 
\[ P = \{ a_i \geq 0, \sum_{i=0}^s a_i(x_i-y_i) \leq 1 \} \]
For a face $\tau$ of $\sigma$ we have that $\tau$ is contracted if and only if $\deg \D_1 \cap \tau \neq \emptyset$. We see that 
\[ \deg \D_1= \{ (a_1,...,a_n) \in N_\Q | a_i \geq 0, \sum_{i=0}^s a_i(x_i-y_i)  \geq 1 \} \]
A face $\tau$ of $\sigma$ is given by a arbitrary subset $S$ of $\{1,...,n\}$:
\[ \tau = \Cone(e_i | i \in S \subset \{ 1,...,n \}) = \{ (a_1,...,a_n) \in N_\Q | a_i \geq 0, a_i = 0 \text{ for } i \notin I \} \]
We see that $\tau$ is not contracted if and only if $S \subset \{s+1....,n \}$.

The last case is if there exists at least one component such that $u_1-u_2$ is positive and at least one component such that it is negative. We have that now $\D_1$ and $\D_2$ will have distincti full-dimensional tailcones $\sigma_1$ and $\sigma_2$, with $\sigma=\sigma_1 \cup \sigma_2$. There are no empty coefficients in any of the polyhedral divisors, thus all faces of $\sigma_1,\sigma_2$ are contracted. 

The above might seem quite technical and non-illuminating, however if we go to the Klyachko perspective and reformulate the above in terms of the filtrations we get the following.

\begin{proposition} \label{proposition:tanktwocycles}
Let $X \simeq \p(\E)$ be a toric vector bundle of rank two on a smooth toric variety $X_\Sigma$. Then there is a divisorial fan for $X$ with tailfan $\Sigma'$ being a subdivision of $\Sigma$. For a cone $\tau \in \Sigma'$ the corresponding orbit is not contracted by the map $\widetilde{X} \to X$ if and only if $\tau \in \Sigma$ and for any $v \in \tau(1)$ there doesn't appear a one-dimensional space in the filtration $E^v(j)$.
\end{proposition}
\begin{proof}
This is just a reformulation of the above, when we recall that the condition $u_i=z_i$ is equivalent to $E^{v_i}(j)$ jumps directly from $0$ to $E$.
\end{proof}

\begin{example}
The $T$-variety in Example \ref{example:bunldep1p1} can be obtained as follows. Start with $\p^1 \times \p^1$, considered as the toric variety with rays $\pm e_i, i=1,2$ and consider a rank two toric vector bundle on it. The filtrations are given by 
 \begin{align*} E^{e_1}(j) = \begin{cases} 
E &\text{ if } j \leq 0 \\
[0] &\text{ if } 0 < j \leq 1 \\
0 &\text{ if } 1 < j
\end{cases},  E^{e_2}(j) = \begin{cases} 
E &\text{ if } j \leq 0 \\
[1] &\text{ if } 0 < j \leq 1 \\
0 &\text{ if } 1 < j
\end{cases} \end{align*}  
 \begin{align*} E^{-e_1}(j) = \begin{cases} 
E &\text{ if } j \leq 0 \\
[\infty] &\text{ if } 0 < j \leq 1 \\
0 &\text{ if } 1 < j
\end{cases}, E^{-e_2}(j) = \begin{cases} 
E &\text{ if } j \leq 0 \\
0 &\text{ if } 1 < j
\end{cases} \end{align*}  
where we by $[p]$ mean the one-dimensional vector space in $E$ for which the corresponding point in $\p(E)$ is $p$. Then one can easily check that we obtain the data given in Figure \ref{figure:bundloeoverp1p1} and \ref{figure:degbundlep1p1}.
\end{example}

Given $\E$  of rank two on $X_\Sigma$ we partition the cones $\sigma$ in $\Sigma$ into three subsets as above. For any cone $\sigma$ we denote the  characters associated to $\E|_{U_\sigma}$ by $u_1^\sigma$ and $u_2^\sigma$. We say that $\sigma \in H$ if $u_1^\sigma = u_2^\sigma$ on $\sigma$.  We say that $\sigma \in I$ if $u_1^\sigma-u_2^\sigma >0$ and $u_2^\sigma-u_1^\sigma >0$ both intersects $\sigma$. Lastly $\sigma \in J$ if $u_1^\sigma \geq u_2^\sigma$ (or $\leq$) on $\sigma$, but they are not equal. 

\begin{proposition} \label{proposition:rank2tvb}
For a rank two toric vector bundle $\E$ on the smooth toric variety $X_\Sigma$ we have that for $X=\p(\E)$ the cycles $R_k,V_k,T_k$ for $k<n$ corresponds bijectively to the following sets
\[ R_k \leftrightarrow \Sigma(n-k+1) \cap H \]
\[ V_k \leftrightarrow (\Sigma(n-k+1) \cap J) \cup (\Sigma(n-k) \cap J)  \cup 2 (\Sigma(n-k) \cap H)  \]
\[ T_k \leftrightarrow (\Sigma(n-k+1) \cap I) \cup (\Sigma(n-k) \cap J) \cup 2 (\Sigma(n-k) \cap I) \]
In particular, by summing these numbers we get $r_k+v_k+t_k = \#\Sigma(n-k+1) + 2 \#\Sigma(n-k)$ for $k<n$.

When $k=n$ we have that the correspondences for $R_n$ and $T_n$ still hold, however an element of $V_n$ corresponds to $(\Sigma(1) \cap J) \cup P$ and $r_n+v_n+t_n = \# \Sigma(1) + \#P$.
\end{proposition}
In particular the number of generators in the surjection (\ref{surjection}) does not depend on the bundle $\E$ except when $k=n$. 
\begin{proof}
This is an application of Lemma \ref{lemma:downgradecycle} together with the description of $\p(\E)$ as a $T$-variety, since locally a toric vector bundle is given as a toric downgrade.

The statement on $R_k$ is immediate from Proposition \ref{proposition:tanktwocycles}.

If $\sigma \in \Sigma(n-k+1) \cap J$ then $D_2^\sigma = \nabla_1 \otimes v_1 + \Delta_2 \otimes v_2$ where $\nabla_1$ is a sum $\tau+P$, where $P$ is a polytope and $\tau$ a cone not contracted by $r$, and $\Delta_2$ equals $\emptyset$. Recall that 
\[ \nabla_1 = \{ (a_1,...,a_n) \in N_\Q | a_i \geq 0, \sum_{i=0}^s a_i(x_i-y_i)  \leq 1 \} \]
The intersection of the hyperplane $\sum_{i=0}^s a_i(x_i-y_i)  \leq 1$ with $\sigma$ is a face of dimension $n-k$ contained in the relative interior of $\sigma$, with tailcone $\tau$, thus giving an element of $V_k$. If $\sigma \in \Sigma(n-k) \cap J$ then the corresponding $\nabla_1$ itself will give an element of $V_k$. If $\sigma \in \Sigma \cap H$ then there will be two different fibers, each containing $\sigma$ as a face, however both will have empty coefficients somewhere, thus these fibers are special fibers, hence they contribute twice to $V_k$.

If $\sigma \in \Sigma(n-k+1) \cap I$ then in $\Sigma'$ $\sigma$ will be subdivided by intersecting with a hyperplane intersecting the interior of $\sigma$. The intersection of $\sigma$ with this hyperplane is a cone of dimension $n-k$ which will be contracted by $r$, giving an element of $T_k$. If $\sigma \in \Sigma(n-k) \cap J$ then by the above $\sigma$ is contracted thus it defines an element of $T_k$. If $\sigma \in \Sigma(n-k) \cap I$ then in $\Sigma'$ it will be subdivided into two different cones of dimension $n-k$, both of which will be contracted by $r$.

For $k=n$ we have the short exact sequence
\[ 0 \to \Z^P/\Z \oplus M \to Z^{V_n \cup R_n} \to \Picc(\p(\E)) \to 0 \]
Implying that $\# \Sigma(1) + \# P = v_n + r_n$. Moreover we see that as above
\[ R_n \leftrightarrow \Sigma(1) \cap H \]
However an element of $V_n$ (which is simply any vertex of a special fiber) corresponds to either $\Sigma(1) \cap J$ (as above) or to the vertex $0$ which is a vertex of $S_p$ for any $p \in P$. 
\end{proof}

\begin{remark}
If $\E$ is a direct sum of line bundles, then $\p(\E)$ is itself a toric variety $X_{\Sigma'}$ and the presentation of $\p(\E)$ as a $T$-variety is simply a toric downgrade. In this case Proposition \ref{proposition:rank2tvb} follows from Corollary \ref{corollary:downgrade} and the description of the fan $\Sigma'$ \cite[Proposition 7.3.3]{CLS}. 
\end{remark}

\begin{remark}
Consider once again the varieties in Example \ref{example}. We see that for $\E$ and $\F$ the numbers $r_k,v_k,t_k$ are given by Table \ref{table}. We see that $r_k+v_k+t_k$ is independent of whether we use $\E$ or $\F$, as predicted by Proposition \ref{proposition:rank2tvb}.

\begin{table} 
\resizebox{\columnwidth}{!}{%
  \begin{tabular}{|l|l|l|l|l|l|l|l|l|l|l|l|l|}
    \hline
    & $r_2$ & $v_2$ & $t_2$ & $r_2+v_2+t_2$ & $r_1$ & $v_1$& $t_1$ & $r_1+v_1+t_1$ & $r_0$& $v_0$& $t_0$ & $r_0+v_0+t_0$ \\
    \hline
    $\p(\E)$ & 3 & 0 & 2 &5 & 7 & 1 & 1 & 9 & 4 & 2 & 0  & 6\\
    \hline
    $\p(\F)$ & 5 & 0 & 0 & 5& 4 & 5 & 0 & 9 & 1 &  5 & 0   & 6\\
    \hline
  \end{tabular}%
}
  \caption{The numbers $r_i,v_i,t_i$ for Example \ref{example}}
  \label{table}
\end{table}
\end{remark}

\bibliography{ref}

\newcommand{\etalchar}[1]{$^{#1}$}
\begin{thebibliography}{DELV11}

\bibitem[AH06]{AH}
Klaus Altmann and J\"urgen Hausen.
\newblock Polyhedral divisors and algebraic torus actions.
\newblock {\em Math. Ann.}, 334(3):557--607, 2006.

\bibitem[AH08]{AltH}
Klaus Altmann and Georg Hein.
\newblock A fansy divisor on {$\overline{M}_{0,n}$}.
\newblock {\em J. Pure Appl. Algebra}, 212(4):840--850, 2008.

\bibitem[AHS08]{AHS}
Klaus Altmann, J\"{u}rgen Hausen, and Hendrik S\"{u}ss.
\newblock Gluing affine torus actions via divisorial fans.
\newblock {\em Transform. Groups}, 13(2):215--242, 2008.

\bibitem[AIP{\etalchar{+}}12]{TVAR}
Klaus Altmann, Nathan~Owen Ilten, Lars Petersen, Hendrik S\"u\ss, and Robert
  Vollmert.
\newblock The geometry of {$T$}-varieties.
\newblock In {\em Contributions to algebraic geometry}, EMS Ser. Congr. Rep.,
  pages 17--69. Eur. Math. Soc., Z\"urich, 2012.

\bibitem[AP12]{AP}
Klaus Altmann and Lars Petersen.
\newblock Cox rings of rational complexity-one {$T$}-varieties.
\newblock {\em J. Pure Appl. Algebra}, 216(5):1146--1159, 2012.

\bibitem[CLS11]{CLS}
David~A. Cox, John~B. Little, and Henry~K. Schenck.
\newblock {\em Toric varieties}, volume 124 of {\em Graduate Studies in
  Mathematics}.
\newblock American Mathematical Society, Providence, RI, 2011.

\bibitem[DELV11]{DELV}
Olivier Debarre, Lawrence Ein, Robert Lazarsfeld, and Claire Voisin.
\newblock Pseudoeffective and nef classes on abelian varieties.
\newblock {\em Compos. Math.}, 147(6):1793--1818, 2011.

\bibitem[FL17]{FL}
Mihai Fulger and Brian Lehmann.
\newblock Positive cones of dual cycle classes.
\newblock {\em Algebr. Geom.}, 4(1):1--28, 2017.

\bibitem[FMSS95]{FMSS}
W.~Fulton, R.~MacPherson, F.~Sottile, and B.~Sturmfels.
\newblock Intersection theory on spherical varieties.
\newblock {\em J. Algebraic Geom.}, 4(1):181--193, 1995.

\bibitem[FS97]{FS}
William Fulton and Bernd Sturmfels.
\newblock Intersection theory on toric varieties.
\newblock {\em Topology}, 36(2):335--353, 1997.

\bibitem[Ful98]{Fulton}
William Fulton.
\newblock {\em Intersection theory}, volume~2 of {\em Ergebnisse der Mathematik
  und ihrer Grenzgebiete. 3. Folge. A Series of Modern Surveys in Mathematics
  [Results in Mathematics and Related Areas. 3rd Series. A Series of Modern
  Surveys in Mathematics]}.
\newblock Springer-Verlag, Berlin, second edition, 1998.

\bibitem[GfKZ94]{GKZ}
I.~M. Gel\cprime~fand, M.~M. Kapranov, and A.~V. Zelevinsky.
\newblock {\em Discriminants, resultants, and multidimensional determinants}.
\newblock Mathematics: Theory \& Applications. Birkh\"{a}user Boston, Inc.,
  Boston, MA, 1994.

\bibitem[HS10]{HS}
J\"{u}rgen Hausen and Hendrik S\"{u}\ss.
\newblock The {C}ox ring of an algebraic variety with torus action.
\newblock {\em Adv. Math.}, 225(2):977--1012, 2010.

\bibitem[IS11]{IS}
Nathan~Owen Ilten and Hendrik S\"{u}ss.
\newblock Polarized complexity-1 {$T$}-varieties.
\newblock {\em Michigan Math. J.}, 60(3):561--578, 2011.

\bibitem[Kap93]{Kapranov}
M.~M. Kapranov.
\newblock Chow quotients of {G}rassmannians. {I}.
\newblock In {\em I. {M}. {G}el\cprime fand {S}eminar}, volume~16 of {\em Adv.
  Soviet Math.}, pages 29--110. Amer. Math. Soc., Providence, RI, 1993.

\bibitem[Kly89]{Klyachko}
A.~A. Klyachko.
\newblock Equivariant bundles over toric varieties.
\newblock {\em Izv. Akad. Nauk SSSR Ser. Mat.}, 53(5):1001--1039, 1135, 1989.

\bibitem[LLM15]{Topology}
Antonio Laface, Alvaro Liendo, and Joaquín Moraga.
\newblock On the topology of rational t-varieties of complexity one, 2015.

\bibitem[Oda78]{Oda}
Tadao Oda.
\newblock {\em Torus embeddings and applications}, volume~57 of {\em Tata
  Institute of Fundamental Research Lectures on Mathematics and Physics}.
\newblock Tata Institute of Fundamental Research, Bombay; by Springer-Verlag,
  Berlin-New York, 1978.
\newblock Based on joint work with Katsuya Miyake.

\bibitem[PS11]{PS}
Lars Petersen and Hendrik S\"{u}ss.
\newblock Torus invariant divisors.
\newblock {\em Israel J. Math.}, 182:481--504, 2011.

\bibitem[{Sco}13]{T-Curves}
Geoffrey {Scott}.
\newblock {Torus Invariant Curves}.
\newblock {\em arXiv e-prints}, page arXiv:1304.3822, Apr 2013.

\end{thebibliography}
\bibliographystyle{alpha}
\end{document}